\documentclass[12pt]{amsart}
\usepackage[utf8]{inputenc}
\usepackage{amsmath, amsfonts, amssymb,amscd, amstext, amsthm, mathtools, bezier}
\DeclareMathOperator{\ISO}{Isom}
\DeclareMathOperator{\PSL}{PSL}
\DeclareMathOperator{\F}{F}
\DeclareMathOperator{\tr}{trace}
\DeclareMathOperator{\SG}{\Gamma}
\DeclareMathOperator{\gir}{girth}
\DeclareMathOperator{\sys}{sys}
\DeclareMathOperator{\dist}{dist}
\DeclareMathOperator{\per}{Perm}
\DeclareMathOperator{\supp}{supp}
\DeclareMathOperator{\vol}{vol}
\DeclareMathOperator{\AS}{\mathcal{AS}}
\newtheorem{definition}{Definition}
\newtheorem{theorem}{Theorem}[section]
\newtheorem{lemma}[theorem]{Lemma}
\newtheorem{proposition}[theorem]{Proposition}
\newtheorem{corollary}[theorem]{Corollary}
\newtheorem*{rem}{Remark}
\newtheorem*{notation*}{Notation}

\begin{document}

\title{Asymptotic properties of the set of systoles of arithmetic Riemann surfaces} 
\author{Cayo D\'oria}
% \thanks{Supported by a CNPq research grant.}
\address{
Universidade Federal de Goi\'as \\
Instituto de Matem\'atica e Estat\'istica \\
Rua Jacarand\'a -- Ch\'acaras Calif\'ornia, 74001-970. Goi\^ania - GO, Brazil.
}
\email{cayodoria@ufg.br}

\begin{abstract}
The purpose this article is to try to understand the mysterious coincidence between the asymptotic behavior of the volumes of the Moduli Space of closed hyperbolic surfaces of genus $g$ with respect to the Weil-Petersson metric and the asymptotic behavior of the number of arithmetic closed hyperbolic surfaces of genus $g$. If the set of arithmetic surfaces is well distributed then its image for any interesting function should be well distributed too. We investigate the distribution of the function systole. We give several results indicating that the systoles of arithmetic surfaces can not be concentrated, consequently the same holds for the set of arithmetic surfaces. The proofs are based on different techniques: combinatorics (obtaining regular graphs with any girth from results of B.~Bollobas and constructions with cages and Ramanujan graphs), group theory (constructing finite index subgroups of surface groups from finite index subgroups of free groups using results of G.~Baumslag) and geometric group theory (linking the geometry of graphs with the geometry of coverings of a surface).        
\end{abstract}

\maketitle
\section{Introduction}

For any integer $g \geq 2$, let $\mathcal{M}_g$ be the moduli space of orientable, closed, hyperbolic surfaces of genus $g$. There exists a continuous function in $\mathcal{M}_g$ called \emph{systole}, denoted by $\sys$ and given by $\sys(S)=\min \{ l(\gamma) ; \gamma \subset S \, \mbox{is a closed geodesic} \}$, where $l(\gamma)$ denotes the length of the geodesic. For any $S \in \mathcal{M}_g$, $\sys(S)$ coincides with twice of its injectivity radius. Hence by Mumford's Compactness Theorem, for any positive number $\mu$, the set $\mathcal{M}_{g,\mu} = \{ S \in \mathcal{M}_g \, | \, \sys(S) \geq \mu \}$ is  compact for any $g$.  Moreover, there exists a positive constant $A$ which does not depend on $g$ such that for any $S \in \mathcal{M}_g$, $\sys(S) \leq 2 \log(g) + A$. 

For more relations between $sys$ and topological/geometric properties of $\mathcal{M}_g$, see for example \cite{Schaller99} and the references therein. 

In $\mathcal{M}= \displaystyle\cup_{g \geq 2} \mathcal{M}_g$ we can define an equivalence relation $\sim$, where $S_1 \sim S_2$ if there exists $S \in \mathcal{M}$ which covers $S_1$ and $S_2$. In this case $S_1$ and $S_2$ are called \emph{commensurable}. Any equivalence class $\mathcal{C} \in \mathcal{M} / \sim $ is called a commensurability class.

We have a special subset $\AS = \{\mbox{arithmetic compact hyperbolic surfaces}\}$. These surfaces have remarkable geometric properties, for example, the Hurwitz upper bound for the cardinality of the isometry group of a compact hyperbolic surface is attained only by arithmetic surfaces (see \cite{Bel97}) and there exist sequences $\{ S_i \}$ of compact hyperbolic surfaces with genus of $S_i = g_i$ arbitrarily large and $\sys(S_i) \gtrsim \eta \log(g_i)$ (notice that necessarily $\eta \leq 2).$ The best known constant is $\eta = \frac{4}{3}$  which is realized by congruence coverings of an arithmetic hyperbolic surface \cite{BS94}. In \cite{Schaller97}, Paul Schaller considers the following Hypothesis: ``The definition of the best metric should be chosen such that (some) arithmetic surfaces are among the surfaces with the best metric". 

The set $\AS$ is closed with respect to $\sim$, i.e. if $S_1 \in \AS$ and $S_2 \sim S_1$ then $S_2 \in \AS$. We will recall the precise definition of arithmetic hyperbolic surfaces in Section \ref{geoappl}. Now the important fact which we need to recall is that for any $g$ the set $\AS \cap \mathcal{M}_g$ is finite. Moreover, it is possible to quantify this finiteness. Indeed, for any $g \geq 2$, let $\AS(g)$ be the number of arithmetic closed hyperbolic surfaces of genus $g$. The asymptotic growth of $\log \AS(g)$ was given in \cite{BGLS10}, they showed that $$\lim_{g \to \infty} \frac{\log \AS(g)}{g \log g} = 2.$$

On the other hand, if we consider the Weil-Petersson metric on $\mathcal{M}_{g}$, in \cite{ST01} it was shown that the asymptotic growth of $\log \vol(\mathcal{M}_g)$ with respect to this metric is given by $$ \lim_{g \to \infty} \frac{\log \vol(\mathcal{M}_g)}{g \log g}= 2.$$ 

The main motivation of this work is to try to understand this amazing coincidence. The natural question is: ``How is the set of arithmetic surfaces distributed in $\mathcal{M}$''? 

In this paper we are interested in understanding how the set $\AS$ is distributed in $\mathcal{M}$ in terms of the functions $\log \sys$ and $\sys$. More precisely, what can we say about the set $\log \sys(\AS)$ or $\sys(\AS)$? 

Firstly, we can say that there exists a sequence $g_i \to \infty$  such that $\AS \cap \mathcal{M}_{g_i}$ has diameter arbitrary large when $i$ goes to infinity with respect to Teichmuller and Weil-Petersson metrics. 

Indeed, the function $\log \sys(S)$ is a $1-$Lipschitz proper function with respect to the Teichmuller metric on $\mathcal{M}_g$ (see \cite[Thm.~6.4.3]{Buser92}), and in a recent paper \cite{WU19}, Yunhui Wu showed that there exists a constant $k>0$ which does not depend on $g$ such that the function $\sqrt{\sys}$ is $k-$Lipschitz with respect to Weil-Petersson metric. 

We consider a sequence of congruence coverings $S_i$ of a given arithmetic hyperbolic surface $S$ mentioned above and for the same degree we consider a sequence of coverings $T_i$ of $S$ such that $\sys(T_i) = \sys(S)=s$. If $g_i$ is the genus of $S_i$ and $T_i$, since $\sys(S_i) \geq \frac{4}{3} \log(g_i) - B$ for some $B>0$ which does not depend on $i$, then the diameter of $\AS \cap \mathcal{M}_{g_i} \geq d(S_i,T_i)$ and in both metrics this distance is bounded below by a function which goes to infinity when $i$ goes to infinity. 

Our first result shows that in terms of $\log \sys$, any sequence of hyperbolic surfaces of different genera and systole uniformly bounded below is asymptotically aproximated by a sequence of surfaces in a commensurability class which contains a surface of genus $2$. 
 
\begin{theorem} \label{teoA}
Let $\mathcal{C}$ be a commensurability class containing a surface of genus $2$. Let $\left( S_g \right) $ be a sequence in $\mathcal{M}$ with $S_g \in \mathcal{M}_{g,\mu}$ for some uniform constant $\mu>0$. Then there exists a sequence $\left( T_g \right) $ with $T_g \in \mathcal{C} \cap \mathcal{M}_g$ such that $$ \log \sys(T_g) \asymp \log \sys(S_g) .$$  
\end{theorem}

In this paper, the relation $f(x) \asymp g(x) $  for two positive functions means that there exists a constant $L >0 $ such that $\frac{1}{L} g(x) \leq f(x) \leq L g(x)$ for all $x$.

In Section \ref{geoappl} we will see that there exist arithmetic surfaces of genus $2$. The following corollary says that, in terms of $\log \sys, $ the set of arithmetic surfaces is asymptotically, well distributed in the thick part of $\mathcal{M}$.
 
\begin{corollary}
Let $\left( S_g \right) $ be a sequence in $\mathcal{M}$ with $S_g \in \mathcal{M}_{g,\mu}$ for some uniform constant $\mu>0$. Then there exists a sequence $\left( A_g \right) $ with $A_g \in \AS \cap \mathcal{M}_g$ such that $$ \log \sys(A_g) \asymp \log \sys(S_g) .$$ 
\end{corollary}

The second result is about the multiplicity of a systole in a commensurability class. For infinitely many values of systoles, the multiplicity is infinite.
\begin{theorem} \label{teoB}
 Let $\mathcal{C}$ be a commensurability class containing a surface of genus $2$. Then there exists an unbounded sequence of real numbers $s_1<s_2 < \cdots$, such that for any $s_j$ the class $\mathcal{C}$ contains an infinite set of surfaces with systole equal to $s_j$.
\end{theorem}

Again we have a corollary for the set of arithmetic surfaces. Recall that a function is $n:1$ for some $n \in \mathbb{N}$ if the pre-image of any point has cardinality at most $n$.
\begin{corollary}
 The function $\sys: \AS \rightarrow \mathbb{R} $ cannot be $n:1$ for any $n$.
\end{corollary}

The last theorem gives us an additive subset contained in $\AS$.

\begin{theorem} \label{teoC}
There exists a constant $s=2\cosh^{-1}(1+\sqrt{2})=3.057...$, such that for any $k \in \mathbb{N}$ we can find an arithmetic surface $S_k \in \AS$ such that $\sys(S_k)=ks$.
\end{theorem}

Note that the Theorem \ref{teoC} shows that the set of systoles of arithmetic hyperbolic surfaces is well distributed on the real line.
The proofs of Theorems \ref{teoA}, \ref{teoB}, \ref{teoC} are based on the same principle, we fix a surface group of genus 2 and construct sequences of subgroups of finite index, whose geometric properties which allow us to apply the Milnor-Schwarz Lemma to obtain a relation between the geometry of the group with some word metric and the geometry of the action by the group on hyperbolic space. The systole of the surface quotient is related to the minimal length of any non trivial element in the subgroup. In order to obtain these subgroups we use two classical results of very distinct nature. 

In \cite{Baums62}, Gilbert Baumslag showed that surface groups are residually free, and five years later his brother Benjamim Baumslag generalised this result in \cite{Baumslag67} showing that surface groups are fully residually free. A group $G$ is fully residually free if for any finite set $X \subset G$ which does not contain the identity there exists a normal subgroup $N \lhd G$ such that $N \cap X = \emptyset$ and $G / N$ is a free group. Since the surface group of genus 2 is an amalgamated product of two copies of the free group of rank $2$, we can use the proof of Benjamin Baumslag to get for any ball in the surface group with a word metric a normal subgroup which intersects this ball trivially and the quotient is a free group of rank $2$. We will use the notations $\SG_2$ for the surface group of genus $2$ and $\F_2$ for the free group of rank $2$.

Using the fully residually freedom of $\SG_2$ we can find subgroups of this group from subgroups of $\F_2$ of finite index. But any finite index subgroup of $\F_2$ can be viewed geometrically as a connected $4$-regular graph on finite set of vertices (a graph is $k$-regular if any vertice has degree $k$) and the converse is true (see Section \ref{4reg&Schreier}). Finding subgroups of free groups with a given minimal length is equivalent to obtaining finite $4$-regular graphs with a given girth (the girth of a graph is the the minimal length of a circuit on it). In \cite{BOL80}, Bollobas gave a probabilistic model for the set of regular graphs such that the random variable girth has a given asymptotic behavior, which says as a corollary that for any $g$ and $n$ sufficiently large, there exist connected $4$-regular regular graphs on $n$ vertices of girth $g$.

The results of Bollobas and Baumslag allow us to translate the geometric problem to a combinatorial problem which we then solve. Now we will give an outline of the paper. In Section \ref{4reg&Schreier} we will make the translation of combinatorics to group theory, we will recall the proof of Gross, which says that connected regular graphs of even degree are Schreier graphs and the additional information that we give is that the girth of the graph is equal to the minimal length of any element of the subgroup which gives the corresponding Schreier graph. In Section \ref{construction} we will construct, by using cages and Ramanujan graphs for any $3 \leq g  \leq c\log(n)$ and $n$ sufficiently large, a $4$-regular graph on $n$ vertices with girth $g$,  where $c$ is an absolute constant. This result has an independent interest  because in this particular case it is an uniform version of the Bollobas corollary mentioned above. Besides this, we will use the Bollobas model in order to show the existence of connected finite $4$-regular graphs with given girth and $2$-girth (the $2$-girth of a graph is the second smallest length of a circuit). In Section \ref{baumsconcrete} we will give a quantitative proof of Baumslag's result that $\SG_2$ is fully residually free of rank 2. For a suitable set of generators of $\SG_2$ we will show that for any $k$ there exists an epimorphism $\psi_k: \SG_2 \rightarrow \F_2$ such that the ball of radius $k$ in $\SG_2$ intersects the kernel of $\psi_k$ trivially. In this section we will also construct a sequence of subgroups of finite index in $\F_2$ such that the element of minimal length is $x^k$ and the second smallest length can be arbitrarily large. In Section \ref{geoappl} we will give a precise definition of arithmetic surfaces and combine the combinatorial and group theoretic facts in order to give the proofs of Theorems \ref{teoA}, \ref{teoB} and \ref{teoC}. The link with geometry is the Milnor-Schwarz Lemma, the main ideia of the proofs in this section is to compare the girth of regular graphs with the systole of surfaces.

\section{The interplay between group theory and combinatorics} \label{4reg&Schreier}

Let $G$ be a finitely generated group, let $H$ be a subgroup of $G$, and let $\Sigma = \{ g_1^{\pm 1}, \cdots, g_r^{\pm 1}\}$ be a symmetric set of generators of $G$. The (right) Schreier coset graph for that group, subgroup, and generating set is defined as follows. Its vertex
set is the set of right cosets of $H$ in $G$. For each right coset $H_i$ and each generator $g_j$ there is an edge from $H_i$ to the right coset $H_i g_j$. 

Let $\Omega$ be a graph and $k>1$ a natural number, $\Omega$ is $k-$regular if every vertex of $\Omega$ has degree $k$. We will denote by $E(\Omega)$ the set of edges of $\Omega$ and by $V(\Omega)$ the set of vertices. An $s$-factor of $\Omega$ is a subgraph $K$ of $\Omega$ which is regular of degree $s$ and which contains every vertex of $\Omega$. When the edges of $\Omega$ can be partitioned into $s$-factors, we say that $\Omega$ is $s$-factorable.

\begin{theorem}[Petersen, 1891]\label{petersen}
Every regular graph (connected or not) of even degree is $2$-factorable.
\end{theorem}
\begin{proof}
 See \cite[Corollary 2.1.5]{Diestel00} for a short proof.
\end{proof}

\begin{rem}
Note that if $\Omega$ is $2k$-regular then $\Omega$ necessarily is partitioned by $k$ distinct $2$-factors.
\end{rem}

The proof of the following proposition is a modification of the proof of \\ Theorem 2 in \cite{Gross77}.

Recall that on a finite graph $\Omega$ a circuit is a closed path without repetitions of edges. The \emph{girth} of $\Omega$, denoted by $\gir(\Omega)$ is the length of minimal circuit. If we consider the set $C(\Omega)$ of circuits on $\Omega$ and $L:C(\Omega) \rightarrow \mathbb{Z}_{>0}$ the function which associates for each circuit its length, the set $L(C(\Omega))=\{ l_1 < l_2 < \cdots < l_\omega \}$ is called the length spectrum of $\Omega$ and the \emph{$j$-girth} of $\Omega$ if exists is the number $l_j$, in particular $\gir(\Omega)$ is the $1$-girth of $\Omega$. 

Denote by $\F_2$ as the free group of rank $2$ with a fixed set of generators $A=\{x^{\pm 1},y^{\pm 1}\}$.

\begin{proposition} \label{schreier}
For any connected $4$-regular graph $\Omega$ on $n$ vertices, there exists a subgroup $\Gamma < \F_2$ of index $n$ with the following properties:
\begin{enumerate}
 \item The Schreier graph of cosets of $\Gamma$ on the symmetric generating set $A$ is isomorphic to $\Omega$;
 \item If $g=\gir(\Omega)$ then $g=\min \{ l_{A}(\gamma) \, | \, \gamma \in \Gamma - \{1\} \}$, where $l_A(\gamma)$ denotes the length of any element $\gamma \in \F_2$ with respect to $A$. 
\end{enumerate}
\end{proposition}
\begin{proof}
 Take a minimal circuit $C$ on $\Omega$ and  a vertex $v$ of $C$. Consider a bijection between $V(\Omega)$ and $\{1, \cdots, n\}$ such that $v$ corresponds to $1$. By Theorem \ref{petersen} $E(\Omega)=K_1 \sqcup K_2$ where $K_i$ is a $2$-factor.
 
 Each $K_i$ can be written as a disjoint union of circuits $C_i^1, \cdots, C_i^{t_i}$. For each $C_i^j$ assign an arbitrary orientation and consider the cyclic permutation $\pi_{ij} \in S_{n}$ which corresponds to the cyclic order in which the oriented circuit $C_i^j$ passes through vertices of $\Omega$. Then, for each $i=1,2$ consider the permutation $\pi_i = \pi_{i1} \cdots \pi_{it_i}$ and the homomorphism $\phi: \F_2 \rightarrow S_n$ given by  $\phi(x)=\pi_1$ and $\phi(y)=\pi_2$. 
 
 The subgroup $\Gamma= \{ t \in \F_2 | \,  \phi(t) \cdot 1=1\} $ has index $n$ because $\phi$ is transitive. Indeed, since $\Omega$ is connected, for each $i$ there exists a path joining the vertices $1$ and $i$, but each edge of this path is contained in $K_1$ or $K_2$, and in any case we can always go from a vertex to its adjacent in the path by the action of $\phi(x)^{\pm 1}$ or $\phi(y)^{\pm 1}$. Hence the path induces a word $w$ in $\{x^{\pm 1},y^{\pm 1}\}$ such that $i=\phi(w)(1)$. In particular, the minimal circuit $C$ is represented by a reduced word $W \in \F_2$ such that $l_A(W)=l(C)=g$ and it $\phi(W)(1) = 1$ then $W \in \Gamma$.
 
 To see the isomorphism in part $1$, take $\alpha_j \in \F_2$ such that $\phi(\alpha_j) \cdot 1 = j$ for any $j=1, \cdots, n$ and choose $\alpha_1 = 1$. The Schreier graph of cosets of $\Gamma$ on $A$ has the vertices $\{ \Gamma, \Gamma \alpha_2, \cdots, \Gamma \alpha_n \}$ and the bijection $\Gamma \alpha_j \longleftrightarrow j$ is an isomorphism of graphs. In fact,  $j$ is adjacent to $k$ if and only if $k=\pi_i^{\pm 1}(j)$ and this is equivalent to  $\phi(\alpha_k)(1) = \phi(a \alpha_j)(1)$ for some $a \in A$. But this means that $\Gamma \alpha_j = \Gamma \alpha_k a^{-1}$, i.e, $\Gamma \alpha_j$ and $\Gamma \alpha_k$ are adjacent.
 
 To complete the proof it remains to show that any non-trivial element $\gamma \in \Gamma$ satisfies $l_A(\gamma) \geq g$. In fact, any non-empty reduced word $\omega \in \Gamma$ corresponds to a closed path of length $l_A(\omega)$ in the Schreier graph of cosets of $\Gamma$. But  any closed path in $\Omega$ has length at least $g$ hence by the isomorphism we have $l_A(\omega) \geq g$. 
\end{proof}

\section{Combinatorial results } \label{construction}

\begin{definition}
 Let $G_1, G_2$ be  connected $4$-regular graphs and let $e_{i} \in E(G_{i})$ be edges. We define $G_1 \ast_{e_1,e_2} G_2$ as the graph obtained by cutting the edges $e_1,e_2$ and adding a vertex $v$ that attaches to the ends of $e_1$ and $e_2$.
\end{definition}

% \begin{figure}[h]
% \centering
% \def \svgwidth{5.0cm}
% \input{drawing.pdf_tex} 
% \end{figure}

\begin{rem}
 Note that $G_1 \ast_{e_1,e_2} G_2$ is a connected $4$-regular graph on $|G_1|+|G_2|+1$ vertices.
\end{rem}

\begin{lemma} \label{Lema1}
 Let $G_1, G_2$ be connected $4$-regular graphs, if $g_{i}=\gir(G_i)$ then there exist edges $e_{i} \in E(G_{i})$ such that $\gir(G_1 \ast_{e_1,e_2} G_2) = \min \{g_1,g_2\} $.
\end{lemma}

\begin{proof}
 We can choose a minimal circuit $C_i$ in each $G_i$ and take the edge $e_i$ in the complement of $C_i$ for $i=1,2$. Consider now a circuit $C \subset G_1 \ast_{e_1,e_2} G_2$. If $C \subset G_i \backslash e_i$ for some $i=1,2$ then $l(C) \geq \min \{g_1,g_2\}$.

 Otherwise assume that $C$ passes through $v$. We can suppose that $C$ passes through $v$ once. In this case we can decompose $C=(v,x) \cup L \cup (y,v)$ where $x,y$ are the vertices of some $e_i$ and $L$ is a path in the corresponding $G_i \backslash e_i$ joining $x$ and $y$. Note that the length of $L$ is at least $\min \{g_{1} -1, g_{2}-1\}$, since otherwise $L \cup e_i$ would contain a circuit of length at most $g_{i}-1$ in $G_i$. Hence we have $l(C) \geq  \min \{g_1 -1, g_2-1\} + 2 \geq \min \{g_1, g_2\}$.
 \end{proof}
\begin{rem}
 In view of Lemma \ref{Lema1} above, from now on we will write $$G_1 \ast G_2 := G_1 \ast_{e_1,e_2} G_2$$ for a suitable choice of edges such that $\gir(G_1 \ast G_2)=\min \{g_1,g_2\}$.
\end{rem}

% We want to construct $4$-regular graphs with a prescribed girth. For the general case we have the following proposition due to Erd{\" o}s and Sachs (see \cite{ES63}).
% 
% \begin{proposition} \label{erdos}
% For any $k \geq 2$ there exists a constant $a(k)>0$ such that if $4 \leq l \leq a(k)\log(m)$, then there exists a $k$-regular graph $G_{2m,l}$ on $2m$ vertices and girth $l$. 
% \end{proposition}
%  
% In the particular case of $4$-regular graphs, we obtain a generalization of the Proposition \ref{erdos}. 
 
\begin{proposition} \label{graphs}
 There exist a constant $c>0$ and an index $n_0 \in \mathbb{N}$ such that for all $n \geq n_0$ and any $3 \leq g \leq c \log(n) $ there exists a 4-regular, connected graph $G_{n,g}$ on $n$ vertices with girth $g$.
\end{proposition}
\begin{proof}
By \cite[Thm. A]{LUW97}, for any $g \geq 3$, there exists a 4-regular, connected graph $H_g$ on $n_g$ vertices with girth $g$ such that $n_g \leq a5^{\frac{3}{4}g}$ for some $a>0$ which does not depend on $g$.

On the other hand, for any odd prime $p$ there exists a 4-regular, connected graph $X_p$ on $p(p-1)(p+1)$ vertices with $\gir(X_p) \geq b \log(|X_p|+1)$ for some $b>0$ which does not depend on $p$ (See \cite[Appendix 4]{DSV03}).

Arguing as in \cite[p.47]{BS94}, there exist positive integers $K$ and $n_0$ such that every $m \geq n_0$ can be written as
$$ m = m_1 + \cdots + m_k,$$
for some $k \leq K $, $m_j$ being of the form $p(p-1)(p+1)+1$ and $m_j \geq m^{\frac{1}{2k}}$.

Whenever $g \geq 3$ satisfies $a5^{\frac{3}{4}g} < \min \{\frac{n}{2}, n-n_0 \}$ it follows that $n-n_g > n_0$ and $n-n_g > \frac{n}{2}$. Therefore there are odd primes $p_1, \cdots, p_k$ such that $$ n-n_g = m_1+ \cdots + m_k, $$
where $m_j=p_j(p_j-1)(p_j+1)+1$, $k\leq K$ and $m_j \geq (n-n_g)^{\frac{1}{2k}} \geq \left( \dfrac{n}{2} \right)^{\frac{1}{2K}}$. 

For any $j=1,\cdots,k$ consider the corresponding graph $X_{p_j}$ and define $$G_{n,g}=H_g \ast \left( X_{p_1} \ast \cdots \ast X_{p_k} \right).$$
Note that by induction $|G_{n,g}|=|H_g|+\sum_{j=1}^k |X_{p_j}| + k = n_g + \sum_{j=1}^k m_j = n$.
Moreover, by Lemma \ref{Lema1} we have $$\gir(G_{n,g}) = \min \{g,\gir(X_{p_1}), \cdots, \gir(X_{p_k})\}.$$ 
If we take $c>0$ such that  $g \leq c \log(n)$ implies $g \leq \frac{b}{2K}\log(\frac{n}{2}) $ and $a5^{\frac{3}{4}g} < \min \{\frac{n}{2}, n-n_0 \}$, then $g \leq b \log(m_j) \leq \gir(X_{p_j})$ for any $1 \leq j \leq k$ and hence $\gir(G_{n,g})=g$.
\end{proof}

\begin{rem}
In \cite{ES63}, Erd{\" o}s and Sachs prove, among other theorems, a similar result to Proposition \ref{graphs}. The essential difference is that while the regular graphs with a prescribed girth obtained by them may have any valency, such graphs always have an even number of vertices. 
\end{rem}

\begin{lemma}\label{boll}
 Let $k,l$ be positive integers such that $3\leq k < l$. Then for all $n$ sufficiently large, there exists a connected $4$-regular graph on $n$ vertices with $\gir=k$ and $2$-$\gir > l$. 
\end{lemma}
\begin{proof}
 For each $ n \geq 2$, let $\mathcal{F}_n$ be the set of $4-$regular graphs on $n$ vertices. On $\mathcal{F}_n$ we put a probability measure given by Bollobas in \cite{BOL80}, and we consider for each $i$ the random variable $X_i$ on $\mathcal{F}_n$ given by $X_i(G)=$ number of closed paths in $G$ of length $i$. By \cite[Thm.2]{BOL80}, fixed $p \geq 3$ the random variables $\{X_i\}_{i=3}^p$ are asymptotically independent Poisson random variables with $X_i$ having mean $\mu_i = \frac{3^i}{2i}$, this means that for any list $\left(m_3, \cdots, m_p \right)$ with $m_i \in \mathbb{Z}_{\geq 0}$  we have
\begin{equation} \label{poisson}
 \lim_n \mathbb{P}_n(X_3=m_3, X_4=m_4, \cdots, X_p=m_p)= e^{-(\mu_3 + \cdots + \mu_p )} \prod_{i=3}^p \frac{\mu_i^{m_i}}{m_i!}.
\end{equation}

Now if we fix $k \geq 3$, let $$\Theta_{k,l}^n = \{ G \in \mathcal{F}_n \, | \, X_k(G)=1 \, \, \mbox{and} \, X_j(G)=0 \, \, \mbox{if} \, \,  3 \leq j \leq l, \, \, j \neq k \}.$$ Note that if $G \in \Theta_{k,l}^n$, then $G$ is a $4$-regular graph on $n$ vertices of girth $k$ and $2$-girth at least $l+1$. By (\ref{poisson}), $$\lim_n \mathbb{P}_n(\Theta_{k,l}^n)=e^{-\mu_k}\mu_k > 0.$$

% Since the measure of $\Theta_k^n$ satisfies:
% 
% $$\mathbb{P}_n(X_3=0, X_4=0, \cdots, X_{k-1}=0, X_k > 0, X_{k+1}=0, \cdots, X_{l}=0)$$ 
% $$ = \displaystyle \sum_{m \geq 1} \mathbb{P}_n(X_3=0, X_4=0, \cdots, X_{k-1}=0, X_k = m, X_{k+1}=0, \cdots, X_{l}=0)$$ 
% 
% by Fatou's Lemma we have 
% $$\liminf_{n} \left( \displaystyle \sum_{m \geq 1} \mathbb{P}_n(X_3=0, \cdots, X_{k-1}=0, X_k = m, X_{k+1}=0, \cdots, X_{l}=0)\right)$$  
% $$ \geq \displaystyle \sum_{m \geq 1} \left( \liminf_n \mathbb{P}_n(X_3=0, \cdots, X_{k-1}=0, X_k = m, X_{k+1}=0, \cdots, X_{l}=0)  \right).$$
% On the other hand, by (\ref{poisson})

% $$\displaystyle \liminf_n \mathbb{P}_n(X_k = m, X_j=0, \, j \neq k, \, 3 \leq j \leq l)=\exp \left( - \displaystyle \sum_{i=3}^{l} \mu_i \right)  \frac{\mu_k^{m}}{m!} $$
% 
% Hence $\liminf_n \mathbb{P}_n(\Theta_k^n) \geq \displaystyle \exp \left( - \displaystyle \sum_{i=3}^{l} \mu_i \right) \sum_{m \geq 1} \frac{\mu_k^m}{m!} > \exp \left( - \displaystyle \sum_{i=3,i \neq k}^{l} \mu_i  \right)$.

By \cite[Thm.9.20]{JLR00} asymptotically almost surely the graphs on $\mathcal{F}_n$ are connected, if we consider the subset of $\Theta_{k,l}^n$ where the graphs are connected, then given $k \geq 3$ for any $n$ sufficiently large this subset has positive measure, in particular, there exists a connected graph $G_{k,l} \in \mathcal{F}_n$ with $\gir(G_{k,l})=k$ and $2$-$\gir(G_{k,l})>l$.
\end{proof}

\section{Group theoretic results} \label{baumsconcrete}

% In this section we will use the notation of B.~Baumslag and G.~Baumslag in \cite{Baums62} and \cite{Baumslag67}, respectively. 
Given two elements $h,f \in \F_2$ we say that $h$ \emph{reacts} with $f$ if $l_A(hf) < l_A(h) + l_A(f)$. If $h$ does not react with $f$ we will write $$h \bullet f.$$
The word $u=[x,y]=xyx^{-1}y^{-1}$ is reduced and cyclically reduced. Moreover, if $f \in \F_2$ commutes with $u$ then $f \in \langle u \rangle $. This means that $u$ generates its own centralizer in  $\F_2.$

\begin{lemma} \label{lema2}
 Let $m \geq 3$ be an integer and $\varepsilon_q \in \{\pm 1 \},$ $q=1,2$. Then for any $l>1$ and $\gamma \in \F_2 \backslash \langle u \rangle$ with $l_A(\gamma) \leq l$  it holds  $$ u^{\varepsilon_1 m l} \, \gamma \, u^{\varepsilon_2 m l} = u^{\varepsilon_1 (m-2) l} \bullet \gamma' \bullet u^{\varepsilon_2 (m-2) l} $$  with $\gamma' \neq 1.$ 
\end{lemma}
\begin{proof}
 If $\gamma$ contains the subword $u^{\pm 1}$ at the beginning we can write $\gamma = u^{\pm 1} \bullet \gamma_1 $, with a finite number of steps we can write $\gamma = u^{i_1} \bullet  \tilde{\gamma}$ for some $i_1 \in \mathbb{Z}$ with $|i_1| \leq \frac{l}{4}$. Similarly, $\tilde{\gamma}= \hat{\gamma} \bullet u^{i_2}$ for some $i_2 \in \mathbb{Z}$ with $|i_2| \leq \frac{l}{4}.$ Hence, we can always write $\gamma = u^{i_1} \bullet \hat{\gamma} \bullet u^{i_2}$ with $|i_1|+|i_2| \leq \frac{l}{4}.$
 
 Note that we have the equation 
 $$u^{\varepsilon_1 m l} \gamma u^{\varepsilon_2 m l} =  u^{\varepsilon_1 (m-2) l} \bullet u^{\varepsilon_1 2l +i_1} \, \hat{\gamma} \, u^{\varepsilon_2 2l+i_2} \bullet u^{\varepsilon_2 (m-2) l}.$$
 Indeed, $|i_1|+|i_2| \leq \frac{l}{4}$ and $m \geq 3$ imply that $\varepsilon_q(m-2)l$ and $2\varepsilon_ql+i_q$ have the same signal for $q=1,2$.
 
 Finally, we have that $\gamma'=u^{\varepsilon_1 2l +i_1} \, \hat{\gamma} \, u^{\varepsilon_2 2l+i_2} \neq 1$. Otherwise, $\hat{\gamma} \in \langle u \rangle $ and consequently $\gamma \in \langle u \rangle.$
 \end{proof}

We use the following notation: $$ {}^s \! \delta^t(a_1)a_2 \dots a_{p-1}(a_p)  $$
will be used to denote any of the four expressions obtained from $a_1a_2\cdots a_p$ by deleting or not deleting $a_1$ and $a_p$ independently.

\begin{proposition} \label{prop1}
Let $k$ be any given positive integer. Suppose that $$ \gamma_1, \gamma_{r+1} \in \F_2 \, , \gamma_2, \cdots, \gamma_r, \eta_1, \cdots, \eta_r \in \F_2 \backslash \langle u \rangle. $$ 
Furthermore suppose that $ \sum_{i=1}^{r+1} l_A(\gamma_i) + \sum_{j=1}^r l_A(\eta_j) \leq k$.

Then there exists a constant $d>0$ which does not depend on $k$ such that
$$ {}^s \! \delta^t(\gamma_1) u^{dk} \eta_1 u^{-dk} \gamma_2 u^{dk} \eta_2 u^{-dk} \cdots u^{dk} \eta_{r} u^{-dk} (\gamma_{r+1}) \neq 1. $$
\end{proposition}

\begin{proof}
 Consider the word $$ \scalebox{0.83} {$ w={}^s \! \delta^t(\gamma_1) u^{7k} \eta_1 u^{-7k} \gamma_2 u^{7k} \eta_2 u^{-7k} \cdots u^{7k} \eta_{r} u^{-7k} (\gamma_{r+1}). $} $$
 By Lemma \ref{lema2} we have $u^{7k} \eta_i u^{-7k} = u^{5k} \bullet \eta_i' \bullet u^{-5k}$ with $\eta_i' \neq 1$. 
 Hence,
 $$\scalebox{0.9} {$w= {}^s \! \delta^t(\gamma_1) u^{5k} \bullet \eta_1' \bullet u^{-5k} \gamma_2 u^{5k} \bullet \eta_2' \bullet  \cdots  \bullet u^{-5k} \gamma_r u^{5k} \bullet \eta_r' \bullet u^{-5k} (\gamma_{r+1}) $}. $$

 We can rewrite this word as
 $$ \scalebox{0.9}{$ {}^s \! \delta^t(\gamma_1) u^{5k} \bullet \eta_1' \bullet u^{-2k}  \left( u^{-3k} \gamma_2 u^{3k} \right) u^{2k} \bullet \cdots \bullet u^{-2k} \left( u^{-3k} \gamma_r u^{3k} \right) u^{2k} \bullet \eta_r' \bullet u^{-5k} (\gamma_{r+1})  $} .$$
 Using the Lemma \ref{lema2} again we have $u^{-3k} \gamma_j u^{3k}= u^{-k} \bullet \gamma_j' \bullet u^k$ with $\gamma_j' \neq 1$ for $2 \leq j \leq r$.
 Then we have
  $$\scalebox{0.9}{$ w =  {}^s \! \delta^t(\gamma_1) u^{5k} \bullet \eta_1' \bullet u^{-3k} \bullet  \gamma_2' \bullet u^{3k} \bullet \eta_2' \bullet u^{-3k} \bullet \cdots \bullet  u^{-3k} \bullet \gamma_r'\bullet  u^{3k}  \bullet \eta_r' \bullet u^{-5k} (\gamma_{r+1}) $}.$$
 If we write 
 $$\scalebox{0.9} {$w'=u^{5k} \bullet \eta_1' \bullet u^{-3k} \bullet  \gamma_2' \bullet u^{3k} \bullet \eta_2' \bullet u^{-3k} \bullet \cdots \bullet  u^{-3k} \bullet \gamma_r' \bullet u^{3k}  \bullet \eta_r' \bullet u^{-5k} $}, $$ it follows that $$l_A(w) \geq l_A(w') - l_A(\gamma_1)-l_A(\gamma_{r+1}) > 20k - 2k=18k >0. $$
Hence, if we take $d=7$, then $w$ is not trivial.
\end{proof}

Let $\SG_2$ be the fundamental group of an orientable compact surface of genus $2$. If we consider a new copy $\F_2'$ of $\F_2$ with the set of generators $A'=\{ x'^{\pm 1}, y'^{\pm 1} \}$ and let $v=[y',x'] \in \F_2'$ it is well-known that $\SG_2$ is isomorphic to the quotient $\F_2 \ast \F_2' / \langle\langle u*v^{-1} \rangle\rangle $ where $\F_2 \ast \F_2'$ is the free product. 

Considering the natural monomorphisms $\iota: \F_2 \rightarrow \SG_2$ and $\iota': \F_2' \rightarrow \SG_2$ then we can identify $\F_2, \F_2'$ as subgroups of $\SG_2$. With this identification in mind, we can take $B=A \cup A'$ as a set of generators of $\SG_2$ and let $l_B(\tau)$ denote the length of any element  $ \tau \in \SG_2$ with respect to $B$.

\begin{proposition} \label{hom}
There exists a constant $\epsilon > 0$ such that for every positive integer $k$ there exists an epimorphism $\psi_k: \SG_2 \rightarrow \F_2$ with the following properties:
\begin{enumerate}
 \item $\psi_k(\iota(t))=t$ for every $t \in \F_2;$
 \item $\psi_k(\gamma) \neq 1$ if $1<l_B(\gamma)\leq k;$
 \item $l_A(\psi_k(\omega)) \leq \epsilon k l_B(\omega)$ for all $\omega \in \SG_2$.
\end{enumerate}
\end{proposition}
\begin{proof}
Fixed a positive integer $l$ we define the map $\rho_l: \F_2' \rightarrow \F_2$ given by $\rho_l(x')=u^l y u^{-l}$ and $\rho_l(y')= u^l x u^{-l}$. Since $\iota(u)=u=\rho_l(v)$ it follows that the map $\tilde{\phi_l}:\F_2 \ast \F_2' \rightarrow  \F_2$ given by $\iota$ and $\rho_l$ descends to the quotient and defines a homomorphism $\phi_l: \SG_2 \rightarrow \F_2.$ By the construction $\phi_l(\iota(t))=\iota(t)=t$ for any $l$ and any $t \in \F_2$.

Let $\gamma \in \SG_2$ be a non trivial element with $l_B(\gamma) \leq k$. Now let $d$ be the constant given by Proposition \ref{prop1}, take $l=dk$ and define $\psi_k= \phi_l = \phi_{dk}.$ If $\gamma \notin \F_2$ then when writing $\gamma$ as a reduced word we have 
$$ \gamma = \gamma_1 \ast \eta_1' \ast \gamma_2 \ast \cdots \ast \eta_r' \ast \gamma_{r+1},$$
where $\gamma_i \in \F_2$ and $\eta_j \in \F_2'$ are non-trivial with the possible exception of $\gamma_1$ and $\gamma_{r+1}$. Since in $\SG_2$ we have $u=v$, we can suppose that $\eta_j' \notin \langle v \rangle$ and $\gamma_i \notin \langle u \rangle$ for $2 \leq i \leq r.$

Hence $$\psi_k(\gamma)=\gamma_1 u^{dk} \eta_1 u^{-dk} \gamma_2 u^{dk} \eta_2 u^{-dk} \cdots u^{dk} \eta_{r} u^{-dk} \gamma_{r+1}.$$
where $\eta_j$ is the word $\eta_j'$ with $x'$ replaced by $y$ and $y'$ by $x$. Note that $\eta_j \neq 1$ if and only if $\eta_j' \neq 1$. Furthermore, $\eta_j' \in \langle v \rangle$ if and only if $\eta_j \in \langle u  \rangle.$ We also have 
$$\sum_i l_A(\gamma_i) + \sum_j l_A(\eta_j) = \sum_i l_A(\gamma_i) + \sum_j l_{A'}(\eta_j')=l_B(\gamma) \leq k. $$
Therefore, we can apply Proposition \ref{prop1} to conclude that $\psi_k(\gamma) \neq 1$ if $1 < l_B(\gamma) \leq k.$
Now for any $b \in B$ we have $l_A(\phi_k(b)) \leq 1+8dk \leq 9dk$ and hence if we take $\epsilon = 9d$ we have 
$$ l_A(\psi_k(\omega)) \leq \max_{b \in B} \{ l_A(\psi_k(b)) \} l_B(\omega) \leq \epsilon k l_B(\omega).$$
\end{proof}

Let $X$ be a non-empty set, we denote by $\per(X)$ the group of permutations of $X$. Given $f \in \per(X)$ we denote $\supp(f)=\{x \in X | f(x) \neq x \}$. Let $a < b$ be positive integers, we will use the notation $[a,b]$ for the set $\{a,a+1, \cdots,b\}$.

\begin{lemma} \label{perm}
 Let $k,m$ be positive integers, $l_0=0, n_0=k$, $\tau_0 = id_{\mathbb{N}}$ and $\sigma_0 = (1, 2, \cdots, k)$ a cyclic permutation in $\per(\mathbb{N})$. Then for every $r \geq 1$ there exist integers $n_r > l_r > n_{r-1}$ and permutations $\sigma_r, \tau_r \in \per(\mathbb{N}) $ such that $\supp(\sigma_r) = [n_{r-1}+1,n_r]$, $ \supp(\tau_r) = [l_{r-1}+1,l_r] $, and for any non-zero integer $l$ with $|l| \leq m$ we have:
  \begin{eqnarray}
    x \in \supp(\sigma_{r-1}) & \Rightarrow & \tau_r^l(x) \in \supp(\sigma_{r}), \label{rel1}  \\
  x \in \supp(\tau_r) \backslash \supp(\sigma_{r-1}) & \Rightarrow  & \sigma_r^l(x) \in \supp(\sigma_r) \backslash \supp(\tau_r). \label{rel2} 
 \end{eqnarray}
 
Moreover, we have the following relations for $r \geq 0$:
\begin{eqnarray}
 l_{r+1} = & l_r + (2m+1)(n_r-l_r), \label{eq1} \\
 n_{r+1} = & n_r + (2m+1)(2m)(n_r-l_r). \label{eq2}
\end{eqnarray}
\end{lemma}
\begin{proof}
 We will prove this by induction on $r$.
 
 For $r=1$ we define $l_1=(2m+1)k$, $n_1=l_1+4m^2k$ and $\tau_1 \in S_{l_1}$ as follows. If we identify 
 $$\{1,\cdots,(2m+1)k \} \approx \{ 0, \cdots, 2m \} \times \{1, \cdots, k \} $$ by $(i,j) \mapsto ik + j$ then we define $\tau_1(0,j)= (1,j)$, $\tau_1(2m-1,j)=(2m,j)$, $\tau_1(2a-1,j)=(2a+1,j)$ if $1 \leq a \leq m-1$ and $\tau_1(2b,j)=(2(b-1),j)$ if $1 \leq b \leq m.$ Note that $\supp(\tau_0)=\emptyset $ and by construction $\supp(\sigma_0) \subset \supp(\tau_1)$.

 Now we define $\sigma_1: [k+1, n_1] \rightarrow [k+1, n_1]$ as follows. Again we make an identification $$ \{k+1, \cdots, n_1 \} \approx \{ 0,\cdots,2m \} \times \{1,\cdots,2mk \}.$$
 given by $(i,j) \mapsto k+i(2mk)+j$. In this case we define $\sigma_1$ by the same formulae of $\tau_1$ as above. 
 Note that for any $t \in [1,k]$ we have $\bigcup_{1\leq|l|\leq m} \tau_1^l(t) \subset \{k+1, \cdots, (2m+1)k \}.$ This implies $(\ref{rel1})$ for $r=1$. On the other hand, if $x \in [k+1,l_1]$ then $\sigma_1^l(x) > l_1 $ for any $1<|l|\leq m$ and therefore $\sigma_1^l(x) \in \supp(\sigma_1) \backslash \supp(\tau_1)$ which prove the lemma for $r=1$.
 
 Suppose we have defined $\sigma_1,\tau_1, \cdots, \sigma_r, \tau_r$ with $\supp(\sigma_i) = [n_{i-1}+1,n_i]$, $ \supp(\tau_i) = [l_{i-1}+1,l_i] $, satisfying $(\ref{rel1})$ and $(\ref{rel2})$, and $n_i>l_i>n_{i-1}$ satisfying $(\ref{eq1})$ and $(\ref{eq2})$ for any $1 \leq i \leq r$. We can take $l_{r+1}= l_r+(2m+1)(n_r-l_r)$ which also can be rewritten as $l_{r+1}=n_r+2m(n_r-l_r)$. In this case we make the identification $$ \{ l_r+1, \cdots, l_{r+1}\} \approx  \{ 0, \cdots, 2m \} \times \{1, \cdots, (n_r - l_r) \} $$ given by $(i,j) \mapsto l_r+i(n_r-l_r)+j.$ Now we define $\tau_{r+1}$ on this set by the same formulae as above. 
 
 Now we take $n_{r+1}=n_{r}+(2m+1)(2m)(n_r-l_r)$ and we define $\sigma_{r+1}$ for $\{n_r+1, \cdots, n_{r+1} \}$ making the identification  
 $$ [n_{r}+1,n_{r+1}] \approx [0,2m] \times [1,(2m)(n_r-l_r)] $$
 given by $(i,j) \mapsto n_r+i(2m)(n_r-l_r)+j,$ and let $\sigma_{r+1}$ be given by the same set of formulae as in the first case.

 Note that $n_{r+1}>l_{r+1}>n_r$, by construction $l_{r+1}$ and $n_{r+1}$ satisfy $(\ref{eq1})$ and $(\ref{eq2})$. To finish the proof we need to check $(\ref{rel1})$ and $(\ref{rel2})$ for $\sigma_{r+1}$ and $\tau_{r+1}$. We fix a non-zero integer $l$ with $|l| \leq m$. If $x \in \supp(\sigma_r)$ then $n_r < \tau_{r+1}^l(x) \leq l_r+(2m+1)(n_r-l_r)=l_{r+1} < n_{r+1}$ and therefore $\tau_{r+1}^l(x) \in \supp(\sigma_{r+1})$ which shows $(\ref{rel1})$. To show $(\ref{rel2})$ we take $x \in [n_r+1, l_{r+1}]=\supp(\tau_{r+1}) \backslash \supp(\sigma_r)$ and see that by the definition $\sigma_{r+1}^l(x) \in  [l_{r+1}+1,n_{r+1}] = \supp(\sigma_{r+1})\backslash \supp(\tau_{r+1})$.
\end{proof}

\section{Proof of the main results} \label{geoappl}

Let $\Gamma < \ISO^+(\mathbb{H}^2)=\PSL(2,\mathbb{R})$ be a torsion free discrete group such that $\Gamma \backslash \mathbb{H}^2$ is a compact hyperbolic surface. The following characterization of the arithmeticity in terms of the traces of the elements of $\Gamma$ is due to Takeuchi \cite{Takeuchi75}. Let $\mathcal{L}(\Gamma)=\{ \tr(\gamma) \, | \, \gamma \in \Gamma \}$. The surface $\Gamma \backslash \mathbb{H}^2$ is called arithmetic if
\begin{enumerate}
 \item  $K = \mathbb{Q}(\mathcal{L}(\Gamma))$ is a finite extension of $\mathbb{Q}$ and $\mathcal{L}(\Gamma) \subset \mathcal{O}_K$ the ring of integers of $K$.
 \item If $\phi : K \rightarrow \mathbb{C}$ is any embedding such that $\phi$ restricted to $\left( \mathcal{L}(\Gamma) \right)^2$ is not the identity then $\phi(\mathcal{L}(\Gamma))$ is bounded, where $\left( \mathcal{L}(\Gamma) \right)^2=\{ t^2 \, | \, t \in  \mathcal{L}(\Gamma) \}.$ 
\end{enumerate}

Let $\Delta'$ be the orientation-preserving subgroup of the triangular group $\scriptstyle \Delta(2,3,8)$ generated by reflections in the sides of a hyperbolic triangle with angles $\frac{\pi}{2},\frac{\pi}{3}$ and $\frac{\pi}{8}$, then $\Delta'$ satisfies $1$ and $2$ with $K=\mathbb{Q}(\sqrt{2})$ (see \cite{Takeuchi77}). 
 
In \cite[Theorem 5.2]{Schaller93} P. Schaller showed that the hyperbolic surface of genus $2$  which have the maximal systole in $\mathcal{M}_2$ is arithmetic, its uniformizing group is a finite index subgroup of $\Delta(2,3,8)$ and its systole has length $s=2\cosh^{-1}(1+\sqrt{2})$.  In particular, this shows the existence of an arithmetic hyperbolic surface of genus $2$. 

Let $S$ be a fixed orientable compact hyperbolic surface of genus $2$. There exists a monomorphism $\rho: \SG_2 \rightarrow \ISO^+(\mathbb{H}^2)$ such that $S \simeq \rho(\SG_2) \backslash \mathbb{H}^2.$ We will denote $\gamma \cdot p := \rho(\gamma)(p)$ for any $\gamma \in \SG_2$ and $p \in \mathbb{H}^2$.

By the Milnor-Schwarz Lemma, fixed a point $p \in \mathbb{H}^2$ there exist constants $q,\beta> 0$ which depend only on the geometry of $S$ and the  set of generators $B$ of $\SG_2$ such that
\begin{equation} \label{milnor}
 \frac{1}{q} l_B(\gamma) - \beta \leq \dist(p,\gamma \cdot p) \leq q  l_B(\gamma) + \beta 
\end{equation}
for every $\gamma \in \SG_2$, where $\dist$ means the hyperbolic distance in $\mathbb{H}^2$.

\begin{theorem} \label{main}
 Let $S$ be an orientable compact hyperbolic surface of genus $2$. There exist absolute constants $c>0, n_0 \in \mathbb{N}$ and constants $m_0,L>0$ which depend on $S$ and the set of generators $B$ of \ $\SG_2$ such that for any sequence $a_n$ of positive integers satisfying $3 \leq m_0 \leq a_n \leq c \log(n)$ we can find for each $n \geq n_0$ a covering $S_{n} \rightarrow S$ of degree $n$ such that 
 $$ \frac{1}{L} \sqrt{a_n} \leq \sys(S_n) \leq L a_n.$$ 
\end{theorem}
\begin{proof}
 Let $c,n_0$ be as in the Proposition \ref{graphs}. If we fix  $n \geq n_0$ we can consider the graph $G_{n,a_n}$ with girth $a_n$. Now we use  Proposition \ref{schreier} to get a subgroup $\Gamma_n < \F_2$ of index $n$ such that $$a_n=\gir(G_{n,a_n})= \min_{\gamma \in \Gamma_n \backslash \{1\}} \{ l_A(\gamma) \}.$$ 
 Now we consider the sequence $k_n = \lfloor \sqrt{\epsilon^{-1} a_n} \rfloor $ where $\epsilon$ is the constant from Proposition \ref{hom} and we take $m_0'$ minimal such that $a_n \geq m_o'$ implies $k_n \geq 1$ for any $n$. For each $k_n$ we take the homomorphism $\psi_{k_n}$ from Proposition \ref{hom}. Define the subgroup $\Lambda_n= \psi_{k_n}^{-1}(\Gamma_n) < \SG_2$, since $\psi_{k_n}$ is surjective $\Lambda_n$ has index $n$ in $\SG_2$ and therefore the natural projection $S_n= \rho(\Lambda_n) \backslash \mathbb{H}^2 \rightarrow S$ is a covering of degree $n.$
 
 We need to estimate the systole of $S_n$. Since by Proposition \ref{hom}, $\psi_{k_n}(t)=t$ for any $t \in \F_2$, let $\gamma_0 \in \Gamma_n$ be the element such that $a_n = l_A(\gamma_0)$. Then $\gamma_0 \in \Lambda_n$ and $l_B(\gamma_0) \leq l_A(\gamma_0)=a_n$. Hence, if we consider the closed geodesic induced by $\gamma_0$ in $S_n$ of length $l(\gamma_0)$ we have $$ \sys(S_n) \leq l(\gamma_0)=\inf_{z} \dist(z, \gamma_0 \cdot z ) \leq \dist(p,\gamma_0 \cdot p) \leq q a_n + \beta .$$

 On the other hand, let $D(p)$ be the fundamental domain of Dirichlet centered in $p$ for the action of $\rho(\SG_2)$ on $\mathbb{H}^2$. Then the length of the systole of $S_n$ can be evaluated if we take a lifting of closed geodesic which realize the systole with the initial point in $D(p)$ because the projection of this geodesic in $S$ is a geodesic of the same length. This means that there exist a point $p_0 \in D(p)$ and a non-trivial element $\omega \in \Lambda_n$ such that $\dist(p_0, \omega \cdot p_0 ) = \sys(S_n)$. If $\delta$ is equal to the diameter of $S$ we have $$ \sys(S_n) \geq \dist(p, \omega \cdot p ) - 2 \dist(p,p_0) \geq  \frac{1}{q} l_B(\omega) - \beta - 2\delta.$$ 
 But if $l_B(\omega) < k_n$ then by Proposition \ref{hom} we have $\psi_{k_n}(\omega) \in \Gamma_n \backslash \{1\}$. It follows by the construction of $\Gamma_n$ and by Proposition \ref{hom} again that
 $$ a_n \leq  l_A(\psi_{k_n}(\omega)) \leq \epsilon k_n^2 < a_n $$
 This contradiction implies that $l_B(\omega) \geq k_n \geq \sqrt{\epsilon^{-1} a_n} - 1$. Hence, if we take $\delta'= \frac{1}{q} + \beta + 2\delta $ and $q'= \epsilon^{-2}q^{-1}$ we have
 $$ \sys(S_n) \geq q' \sqrt{a_n} - \delta'. $$
 Now we can choose $m_0''$ minimal such that $t \geq m_0''$ implies $q'-\frac{\delta'}{\sqrt{t}} >0$. If we take $m_0= \max \{m_0',m_0''\}$ then there exists a positive constant $L$ such that for any $a_n \in \mathbb{N}$ with $m_0 \leq a_n  \leq c\log(n)$ we have for all $n\geq n_0$ 
 $$\frac{1}{L} \sqrt{a_n} \leq q' \sqrt{a_n} - \delta' \leq \sys(S_n) \leq qa_n + \beta \leq L a_n.$$
\end{proof}

The proof of Theorem \ref{teoA} follows from the above theorem if we choose for any $M_g \in \mathcal{M}_g$ with $sys(M) \geq \mu>0$ the closest integer $a(M_g) \in [m_0, c \log (g-1)].$ If we apply Theorem \ref{main} for $n=g-1$ and $a_n=a(M_g)$, for $g$ sufficiently large the corresponding covering $S_{g}$ of $S$ has degree $g-1$, therefore $S_g$ has genus $g$ and the uniform lower bound given by $\mu$ guarantee the control of the quotient $\dfrac{\sys(M_g)}{\sys(S_g)}.$    

\begin{corollary}
 Let $S$ be a compact hyperbolic surface of genus $2$. Then for any $g$ sufficiently large $S$ admits a sequence of finite covering $S_{g} \rightarrow S$, where $g=$ genus of  $S_g$ and $$ \sys(S_{g}) \geq C \sqrt{\log(g)} .$$ For some constant $C$ which depends on $S$.  
\end{corollary}

Of course, better bounds are known for some sequences of surfaces as in \cite{BS94}. Note that our estimate applies to any initial surface $S$ of genus 2 and the coverings encompass almost all genera (compare \cite{Buser78}).

% , but our estimate applies to any initial surface $S$ of genus 2 and does not use arithmeticity.

The following theorem gives a proof of Theorem \ref{teoB} with more information about the sequence of systoles which has infinite multiplicity.
\begin{theorem}
 Let $S$ be a compact hyperbolic surface of genus $2$ and let  $L(S)= \{a_1 < a_2 < \cdots \}$ be the length spectrum of $S$ (without multiplicities and in ascending order). Then there exists a subsequence $(a_{i_r})_{r \geq 1} $ such that for any $r,$ there exists a sequence of finite coverings $S_{m,r} \rightarrow S$ with degree $d_m \to \infty$ and $\sys(S_{m,r}) = a_{i_r}$.
\end{theorem}

\begin{proof}
 Let $p_k:=\lceil \epsilon k^2 \rceil $. We can apply Lemma \ref{boll} for $l=p_k$, thus there exists a connected graph $G_k \in \mathcal{F}_n$ with $\gir(G_k)=k$ and $2-\gir(G_k) > p_k$. Now we use Proposition \ref{schreier} to exhibit for any large $n$ a subgroup $\Gamma_n < \F_2$ of index $n$ such that $G_n$ is isomorphic to the Schreier graph of $\Gamma_n$, $k=\min\{l_A(w) | w \in \Gamma_n, \, w \neq 1 \}$ and $\min\{l_A(w) | w \in \Gamma_n, \, l_A(w)>k \} > p_k$. 

Now consider the homomorphism $\psi_{k}$ given by Proposition \ref{hom} and the sequence of subgroups $\Lambda_n=\psi_{k}^{-1}(\Gamma_n)$. We can suppose that $S=\rho(\SG_2) \backslash \mathbb{H}^2$ satisfies (\ref{milnor}). We have a sequence of coverings $S_n= \rho(\Lambda_n) \backslash \mathbb{H}^2 \rightarrow S$ of degree $n \to \infty$ with  $\sys(S_n)=l(\lambda_n)=a_{i_{n}}$, since the systole of $S_n$ is the length of a closed geodesic in $S$. For any $\lambda \in \Lambda_n$ with $\lambda \neq 1$ we have $l_B(\lambda) \geq k$. Indeed, suppose the contrary, i.e. that there exists $\omega \in \Lambda_n$, $\omega \neq 1$ and  $l_B(\omega)<k$. Then $\omega \notin \F_2$ and by the proof of Proposition \ref{hom} we have $l_A(\psi_{k}(\omega))>k$. Hence, $  l_A(\psi_{k}(\omega)) > p_k \geq \epsilon k^2 $ and by Proposition \ref{hom} again we have $$ \epsilon k^2 < l_A(\psi_{k}(\omega)) \leq \epsilon k l_B(\omega) $$  
which give us the desired contradiction. Hence, for any $k$ if we argue as in the proof of Theorem \ref{main}, we have $$  q' k -\delta' = x_k \leq \sys(S_n)=a_{i_n} \leq y_k = q k + \beta $$
for any $n$ sufficiently large. Since the sequence $a_{i_n}$ is contained in the compact interval $[x_k,y_k]$ and the set $L(S)$ is discrete, we have then $a_{i_n}=a_{t_k}$ for infinitely many values of $n$ for some $t_k \in \mathbb{N}$. Note that if we vary $k$ then the set $\{ t_k \}$ cannot be bounded, since it is possible to take a sequence of $k_j's$ with $k_j \to \infty$ and $[x_{k_u},y_{k_u}] \cap [x_{k_v},y_{k_v}] = \emptyset $ whenever $u \neq v$. 

To finish, we take $i_r:=t_{k_r}$. Then we proved above that for any $r \geq 1$ there exists a subsequence of finite coverings of $S$ with constant systole $a_{i_r}$ and unbounded degree.

\end{proof}

The next theorem is a more quantitative version of Theorem \ref{teoC}.

\begin{theorem}
 Let $M^*$ be the arithmetic surface of genus $2$ of maximal systole in $\mathcal{M}_2$ mentioned in the beginning of this section and let $s=\sys(M^*)=2\cosh^{-1}(1+\sqrt{2})$. Then for any $k \in \mathbb{N}$ there exists a finite covering $M_k \rightarrow M^*$ with $\sys(M_k)=ks$ and degree $\leq (uk)^{vk^2} $ for some positive constants $u,v$.
\end{theorem}
 
\begin{proof}
 Let $\alpha \subset M^*$ be a systole of $M^*$. Since $M^*$ is maximal, by \cite[Propostion 2.6]{Schaller93} the curve $\alpha$ is non-separating. We can suppose that the monomorphism $\rho: \SG_2 \rightarrow \ISO^+(\mathbb{H}^2)$ such that $M^* \simeq \rho(\SG_2) \backslash \mathbb{H}^2$ satisfies: $\rho(x)$ represents $\alpha$ and $p \in \mathbb{H}^2$ is projected on a point of $\alpha$.
 
 Now we take $a=\lceil q( ks + \beta +2 \delta) \rceil$  and $b= \lceil \epsilon a^2 \rceil$  where  $\epsilon$ is as in Proposition \ref{hom}, $q,\beta$ are as in (\ref{milnor}), $s,\delta$ are the systole and diameter respectively of $M^*.$ By Lemma \ref{perm}, if we choose $m=r=b$, we have two permutations $\sigma=\sigma_0 \cdot \sigma_1 \cdots \sigma_{b}$ and $\tau=\tau_1 \cdot \tau_2 \cdots \tau_{b}$ in $S_{N_k}$ where $N_k:=n_{b}$. Note that since $i \neq j$ implies $\supp(\sigma_i) \cap \supp(\sigma_j) = \emptyset $, these permutations commute. The same holds for $\tau_i 's$.
 
 If we take the homomorphism $\xi: \F_2 \rightarrow S_{N_k}$ given by $\xi(x)=\sigma$ and $\xi(y)=\tau$ then the subgroup $H_k < \F_2$ defined by $H_k= \{ w \in \F_2 | \xi(w) \cdot 1 = 1 \} $ has index $\leq N_k$ (by the proof of Lemma \ref{perm} we have in fact an equality). Besides, if $w \in H_k \backslash \left< x \right>$ with $l_A(w) \leq b$ then $w=x^{i_1}y^{j_1} \cdots y^{j_t} x^{i_{t+1}}$ with $j_p, i_{p'} \neq 0$ for any $1\leq p \leq t$, $2 \leq p' \leq t$ and $2t-1 \leq \sum_{p=1}^t (|i_p|+|j_p|) +|i_{t+1}| \leq b$. In particular: $t,|i_p|,|j_{p'}| \leq b$. 
 
 On the other hand, $\xi(w)(1)=\sigma^{i_1}\tau^{j_1} \cdots \tau^{j_t} \sigma^{i_{t+1}}(1)$. If we use $(\ref{rel1}),(\ref{rel2})$, the commutativity of $\sigma_i's$ and $\tau_i's$, and the fact that $ \supp(\sigma_{r-1}) \backslash \supp(\tau_{r-1}) \subset \supp(\tau_r)$ successively for $r=0,1,\cdots, t$, we conclude that $\xi(w)(1) > k$. Therefore, 
 \begin{eqnarray}
  \min\{l_A(w) | w \in H_k \, \mbox{and}  \, w \notin \, \left< x \right> \} > b. \label{ineqfinal}
 \end{eqnarray}
 Now we will apply Proposition \ref{hom}. We take the homomorphism $\psi_{a}:\SG_2 \rightarrow \F_2$ and define the group $G_k = \psi_{a}^{-1}(H_k)$. Let $M_k = \rho(G_k) \backslash \mathbb{H}^2 $ be the covering of $M^*$. Since $x^k \in G_k$ and $\rho(x)$ represents the systole of $M^k$, we have $$ \sys(M_k) \leq ks .$$
 Now let $\gamma \subset M_k$  be a closed geodesic $\gamma$ of length $l(\gamma)$. If we repeat the same argument of the proof of Theorem \ref{main}, there exists a non-trivial element $ \gamma \in G_k$ representing this geodesic such that 
 \begin{eqnarray}
 l(\gamma) \geq  \frac{1}{q} l_B(\gamma) - \beta - 2\delta. \label{lowerbound}                                                                                                                                                                                                                                                                                                                                                                                                                                                                                   \end{eqnarray}

If $\gamma$ is not a power of $x$ then $l_B(\gamma) > a $. Indeed, if $l_B(\gamma) \leq a$ then part 2 of Proposition \ref{hom} and the proof of Proposition \ref{prop1} show that $\psi_a(\gamma) \neq 1$ and $\psi_{a}(\gamma) \notin \left<x \right>$ if $ \gamma \notin \left <x \right>$. Hence, using the estimate (\ref{ineqfinal}) and the part 3 of Proposition \ref{hom} we have
 $$\epsilon a^2  \leq  b < l_A(\psi_{k}(\gamma)) \leq \epsilon a l_B(\gamma) \leq \epsilon a^2,$$
 which gives a contradiction. Since $l_B(\gamma) > a \geq q(ks+\beta +2\delta) $ by $(\ref{lowerbound})$ we have $l(\gamma)> ks$. Note that the minimal power of $x$ belonging to $H_k$ is $k$, the element $x^k$ represents the systole of $M_k$ and $\sys(M_k)=ks$. 
 
 To finish the proof we need to estimate the degree $N_k$ of the covering. If we subtract $(\ref{eq1})$ from $(\ref{eq2})$ in Lemma \ref{perm} we have that $n_i-l_i$ is a geometric progression with ratio $t=4b^2$ where $b=m$ and the initial term is $n_0-l_0=k$. Therefore, $n_i-l_i=t^i(n_0-l_0)=(4b^2)^ik$. This implies by $(\ref{eq2})$ that  $N_k=n_b=n_{b-1}+(2b+1)(2b)(4b^2)^{b-1}k=n_{b-2}+((4b^2)^{b-2}+(4b^2)^{b-1})(2b+1)(2b)k= \cdots = n_0 + (1+(4b^2)+(4b^2)^2 + \cdots + (4b^2)^{b-1})(2b+1)(2b)k$. Hence
$$ N_k = \left(1 + \dfrac{(4b^2)^b-1}{4b-1}(2b+1)(2b) \right) k. $$ 
 Since $a \leq Ak$ for some constant $A>0$  there exists a constant $B>0$ such that $b \leq Bk^2$ for any $k\geq 1$. Therefore there exist  constants $u,v>0$ such that $N_k \leq (uk)^{vk^2} $.
\end{proof}

\paragraph{\bf Acknowledgements.} I am grateful to my Ph.D advisor at IMPA, Mikhail Belolipestky, for his constant patience and encouragement. I would like to thank Carlos Matheus, Gisele Teixeira, Pedro Gaspar, Rafael Ponte and Yunhui Wu for their useful comments and interest on this work. I am grateful to the referees for their valuable comments and suggestions. I was supported by CNPq-Brazil.

\bibliographystyle{amsalpha}
\bibliography{Referencias}
\end{document}